\documentclass[10pt]{amsart}
\usepackage{amsmath,amssymb,latexsym,soul,cite}
\usepackage{color,enumitem,graphicx}
\usepackage[colorlinks=true,urlcolor=blue,
citecolor=red,linkcolor=blue,linktocpage,pdfpagelabels,
bookmarksnumbered,bookmarksopen]{hyperref}
\usepackage[english]{babel}
\usepackage[T1]{fontenc}
\usepackage[hyperpageref]{backref}

\usepackage[xetex,a4paper,left=2.62cm,right=2.62cm,top=2.8cm,bottom=2.8cm]{geometry}

\newcommand{\abs}[1]{\left|#1\right|}
\newcommand{\bgset}[1]{\big\{#1\big\}}
\newcommand{\dint}{\ds{\int}}
\newcommand{\ds}[1]{\displaystyle #1}
\newcommand{\M}{{\mathcal M}}
\newcommand{\N}{\mathbb N}
\newcommand{\norm}[2][]{\left\|#2\right\|_{#1}}
\renewcommand{\o}{\text{o}}
\newcommand{\PS}[1]{$(\text{PS})_{#1}$}
\newcommand{\R}{\mathbb R}
\newcommand{\restr}[2]{\left.#1\right|_{#2}}
\newcommand{\set}[1]{\left\{#1\right\}}

\numberwithin{equation}{section}

\newtheorem{theorem}{Theorem}[section]
  \theoremstyle{plain}
  \newtheorem{lemma}[theorem]{Lemma}
  \theoremstyle{plain}
  \newtheorem{proposition}[theorem]{Proposition}
  \theoremstyle{plain}
  \newtheorem{corollary}[theorem]{Corollary}
  \theoremstyle{remark}

\DeclareMathOperator{\supp}{supp}

\newenvironment{enumroman}{\begin{enumerate}

}{\end{enumerate}}

\title[Ground states for anisotropic nonlocal equations]{Ground states for scalar field equations \\ with anisotropic nonlocal nonlinearities}

\author[A.\ Iannizzotto]{Antonio Iannizzotto}
\author[K.\ Perera]{Kanishka Perera}
\author[M.\ Squassina]{Marco Squassina}

\address{Dipartimento di Informatica
\newline\indent
Universit\`a degli Studi di Verona
\newline\indent
C\'a Vignal II
\newline\indent
Strada Le Grazie I-37134 Verona, Italy}
\email{antonio.iannizzotto@univr.it}
\email{marco.squassina@univr.it}

\address{Department of Mathematical Sciences
\newline\indent
Florida Institute of Technology
\newline\indent
150 W University Blvd, Melbourne, FL 32901, USA}
\email{kperera@fit.edu}

\thanks{The third author was partially  supported by 2009 MIUR project:
   ``Variational and Topological Methods in the Study of Nonlinear Phenomena''.}

\begin{document}

\begin{abstract}
We consider a class of scalar field equations with anisotropic nonlocal nonlinearities. 
We obtain a suitable extension of the well-known compactness lemma of Benci and Cerami to 
this variable exponent setting, and use it to prove that the Palais-Smale condition holds at all 
level below a certain threshold. We deduce the existence of a ground state when the variable 
exponent slowly approaches the limit at infinity from below.
\end{abstract}

\maketitle


\section{Introduction and main results}

\noindent
In the present paper we seek {\it ground states}, namely least energy solutions, for the following nonlocal anisotropic scalar field equation:
\begin{equation} \label{ele}
- \Delta u + V(x)\, u = \lambda\, \frac{|u|^{p(x) - 2}\, u}{\int_{\R^N} |u(x)|^{p(x)}\, dx}, \qquad u\in H^1(\R^N).
\end{equation}
Here, $N\geq 2$, $V\in L^\infty(\R^N)$ is a weight function satisfying
\begin{equation} \label{asyv}
\lim_{|x| \to \infty} V(x) = V^\infty > 0,
\end{equation}
and the variable exponent $p \in C(\R^N)$ satisfies
\begin{equation} \label{ppm}
2 < p^- := \inf_{x \in \R^N} p(x) \le \sup_{x \in \R^N} p(x) =: p^+ < 2^\ast,
\end{equation}
\begin{equation} \label{asyp}
\lim_{|x| \to \infty} p(x) = p^\infty
\end{equation}
($2^\ast=2N/(N-2)$ if $N\ge 3$, $2^\ast=\infty$ if $N=2$). Equation \eqref{ele} is the Euler-Lagrange equation for the constrained $C^1$ functional $\restr{J}{\M}$, where we denote for all $u\in H^1(\R^N)$
$$
J(u):=\int_{\R^N} \left(|\nabla u|^2 + V(x)\, u^2\right) dx, \quad I(u):=\inf \Big\{\gamma > 0 : \int_{\R^N} \abs{\frac{u(x)}{\gamma}}^{p(x)} \frac{dx}{p(x)} \le 1\Big\},
$$
and
\[\M := \set{u \in H^1(\R^N) : I(u) = 1},\]
namely $u\in H^1(\R^N)$ solves \eqref{ele} if and only if $u$ is a critical point of $\restr{J}{\M}$ with $J(u)=\lambda$ (see Section 2 below for details). In particular, the ground states of \eqref{ele} are the minimizers of $\restr{J}{\M}$ and the corresponding energy level is
\[
\lambda_1:=\inf_{u\in\M}J(u).
\]
In other terms, \eqref{ele} has a ground state if and only if $\lambda_1$ is attained.
\vskip2pt
\noindent
The constant exponent case $p(x) \equiv p^\infty \in (2,2^\ast)$ of equation \eqref{ele} has been studied extensively for more than three decades (see Bahri and Lions \cite{BL} for a detailed account). Ground states are quite well understood in this case. Set for all $u\in H^1(\R^N)$
\[
I^\infty(u):=\left[\int_{\R^N}|u(x)|^{p^\infty}\frac{dx}{p^\infty}\right]^\frac{1}{p^\infty}, \qquad
\M^\infty := \set{u \in H^1(\R^N) : I^\infty(u) = 1}.
\]
The infimum
\[\widetilde{\lambda}_1 := \inf_{u \in \M^\infty} J(u)\]
is not attained in general. The asymptotic functional
\[J^\infty(u) := \int_{\R^N} \left(|\nabla u|^2 + V^\infty u^2\right) dx, \quad u \in H^1(\R^N)\]
attains its infimum
\begin{equation}
\label{lambda1inft}
\lambda_1^\infty := \inf_{u \in \M^\infty} J^\infty(u)>0
\end{equation}
at a positive radial function $w_1^\infty$ (see Berestycki and Lions \cite{BL1} and Byeon {\it et al.} \cite{BJM}).
Moreover, such minimizer is unique up to translations (see Kwong \cite{K}).
By \eqref{asyv} and the translation invariance of $J^\infty$, one can easily
see that $\widetilde{\lambda}_1 \le \lambda_1^\infty$, and $\widetilde{\lambda}_1$ is attained if this inequality is strict (see Lions \cite{L,L1}).
\vskip4pt
\noindent
In this paper we give sufficient conditions on the weight
$V$ and the exponent $p$ for the existence of a ground state of \eqref{ele}. Precisely, we shall prove the following result:

\begin{theorem}\label{main}
Assume that $V \in L^\infty(\R^N)$ satisfies \eqref{asyv} and $p \in C(\R^N)$ satisfies \eqref{ppm} and \eqref{asyp}. Then $- \infty < \lambda_1 \le \lambda_1^\infty$. If
\begin{equation} \label{ineq}
\lambda_1 < \left(\frac{p^-}{p^\infty}\right)^{2/p^\infty} \lambda_1^\infty,
\end{equation}
then $\lambda_1$ is attained at a positive minimizer $w_1\in C^1(\R^N)$.
\end{theorem}

\noindent
In particular, there exists a positive ground state if $p^\infty = p^-$ and $\lambda_1 < \lambda_1^\infty$ (hence our result is consistent with the constant exponent case). Hypothesis \eqref{ineq} is {\it global}, but it can be assured by making convenient {\it local} assumptions on $p$, for instance when $p(x)$ slowly approaches $p^\infty$ from below as $|x|\to\infty$:

\begin{theorem} \label{local}
Assume that $V \in L^\infty(\R^N)$ satisfies \eqref{asyv} and $p \in C(\R^N)$ satisfies \eqref{ppm} and \eqref{asyp}. Let $\psi\in C^1(\R^+,\R^+_0)$ be a mapping such that $\psi(r) \to \infty$ as $r \to \infty$ and the function $e^{- \psi(|\cdot|)} \in H^1(\R^N)$, and let $R > 0$. Then there exists $a>0$ such that, if
\begin{equation} \label{plocal}
p(x) \le p^\infty - \frac{a}{\psi(|x|)}, \quad |x|\ge R,
\end{equation}
then $\lambda_1$ is attained.
\end{theorem}

\noindent
For example, given $R > 0$, we can find $a>0$ such that, if
\[p(x) \le p^\infty - ae^{- |x|}, \quad |x|\ge R,\]
then $\lambda_1$ is attained. We note that the only assumption on $V$ in Theorem \ref{local} is \eqref{asyv}.
\vskip2pt
\noindent
Finally, we address the problem of {\it symmetry} of minimizers. Apparently, the best we can achieve under the assumption of radial symmetry of the data $V$ and $p$ is axial symmetry of all ground states:

\begin{corollary}\label{symm}
Assume that $N\geq 3$, $V \in L^\infty(\R^N)$ satisfies \eqref{asyv} and $p \in C(\R^N)$ satisfies \eqref{ppm} and \eqref{asyp}, and both are radially symmetric in $\R^N$. Moreover, assume that \eqref{ineq} holds. Then, for every minimizer $w$ there exist a line $L$ though $0$ and a function $\widetilde{w}:L\times\R^+\to\R$ such that
\[w(x)=\widetilde{w}(P_L(x),|x-P_L(x)|), \quad x\in\R^N,\]
where $P_L:\R^N\to L$ denotes the projection onto $L$.
\end{corollary}

\noindent
As in the constant exponent case, the main difficulty is the lack of compactness
inherent in this problem, which originates from the invariance of $\R^N$ under the action
of the noncompact group of translations, and manifests itself in the noncompactness of the
embedding of $H^1(\R^N)$ into the variable exponent Lebesgue space $L^{p(\cdot)}(\R^N)$.
This in turn implies that the manifold $\M$ is not weakly closed in $H^1(\R^N)$ and that
$\restr{J}{\M}$ does not satisfy the Palais-Smale compactness condition (shortly \PS{c}) at all energy
levels $c\in\R$. We will use the concentration compactness principle of Lions (see \cite{L,L1,L2}),
expressed as a suitable profile decomposition for \PS{} sequences of $\restr{J}{\M}$, to overcome
these difficulties. Developing this argument, we will also prove an extension to the variable
exponent case of the compactness lemma of Benci and Cerami \cite[Lemma 3.1]{BC}.
\vskip4pt
\noindent
The paper has the following structure: in Section 2 we introduce the mathematical background
and establish some technical lemmas; in Section 3 we prove that \PS{c} holds for all $c$ below
a threshold level; and in Section 4 we deliver the proofs of our main results.

\section{Preliminaries}

\noindent
We consider the space $H^1(\R^N)$, endowed with the norm defined by
\[\norm{u}^2:=\int_{\R^N}\Big(\abs{\nabla u(x)}^2+V^\infty\abs{u(x)}^2\Big) dx, \quad u\in H^1(\R^N),\]
which is equivalent to the standard norm. Clearly we have $J\in C^1(H^1(\R^N))$ with
\[\langle J'(u),v\rangle = 2 \int_{\R^N} \big(\nabla u \cdot \nabla v + V(x)\, uv\big)\, dx, \quad u, v \in H^1(\R^N).\]
We recall some basic features from the theory of variable exponent Lebesgue space, referring the reader to the book of Diening {\it et al.} \cite{DHHR} for a detailed account on this subject. Let $p\in C(\R^N)$ satisfy \eqref{ppm} and \eqref{asyp}. The space $L^{p(\cdot)}(\R^N)$ contains all the measurable functions $u:\R^N\to\R$ such that
\[\rho(u):=\int_{\R^N}|u(x)|^{p(x)}dx<\infty.\]
This is a reflexive Banach space under the following modified Luxemburg norm, introduced by Franzina and Lindqvist \cite{FL}:
\[\norm[p(\cdot)]{u}:=\inf \Big\{\gamma > 0 : \int_{\R^N} \abs{\frac{u(x)}{\gamma}}^{p(x)} \frac{dx}{p(x)} \le 1\Big\}, \quad u\in L^{p(\cdot)}(\R^N).\]
The following relation will be widely used in our study:
\begin{equation} \label{modnorm}
p^- \min \set{\norm[p(\cdot)]{u}^{p^+},\norm[p(\cdot)]{u}^{p^-}} \le \rho(u) \le p^+ \max \set{\norm[p(\cdot)]{u}^{p^+},\norm[p(\cdot)]{u}^{p^-}}.
\end{equation}
It can be proved noting that, for all $u\in L^{p(\cdot)}$,
\begin{equation}\label{one}
\int_{\R^N} \abs{\frac{u(x)}{\norm[p(\cdot)]{u}}}^{p(x)} \frac{dx}{p(x)}=1.
\end{equation}
Analogously, for all $q>1$ we endow the constant exponent Lebesgue space $L^q(\R^N)$ with the norm
\[\norm[q]{u}^q:=\int_{\R^N}\abs{u(x)}^q\frac{dx}{q}, \quad u\in L^q(\R^N).\]
By \eqref{ppm} and \cite[Theorem 3.3.11]{DHHR}, the embedding $L^{p^-}(\R^N)\cap L^{p^-}(\R^N)\hookrightarrow L^{p(\cdot)}(\R^N)$ is continuous. So, by the Sobolev embedding theorem, also $H^1(\R^N)\hookrightarrow L^{p(\cdot)}(\R^N)$ is continuous. As a consequence, the functional $I:H^1(\R^N)\to\R$ is well defined and continuous. Moreover, reasoning as in \cite[Lemma A.1]{FL}, we see that $I\in C^1(H^1(\R^N)\setminus\{0\})$ with
\[\langle I'(u),v\rangle= \frac{\dint_{\R^N} \abs{\frac{u(x)}{I(u)}}^{p(x) - 2} \frac{u(x)}{I(u)}\: v(x)\, dx}{\dint_{\R^N} \abs{\frac{u(x)}{I(u)}}^{p(x)} dx}, \quad u \in H^1(\R^N) \setminus \set{0},\, v \in H^1(\R^N).\]
In particular, as $1$ is a regular value of $I$, $\M$ turns out to be a $C^1$ Hilbert manifold. By the Lagrange multiplier rule, $u\in\M$ is a critical point of $\restr{J}{\M}$ if and only if there exists $\mu\in\R$ such that
\[J'(u)=\mu I'(u) \quad\mbox{in $H^{-1}(\R^N)$,}\]
that is (recalling that $I(u)=1$), if and only if $u$ is a (weak) solution of \eqref{ele} with $\lambda=\mu/2$. Moreover, testing \eqref{ele} with $u$ yields $J(u)=\lambda$.
\vskip4pt
\noindent
We set
\[\sigma(u):=\max\{\norm[p(\cdot)]{u}^{p^--1},\norm[p(\cdot)]{u}^{p^+-1}\}, \quad u\in L^{p(\cdot)}(\R^N),\]
and we prove the following properties that will be used later:

\begin{lemma}\label{esti}
For all $u,v\in L^{p(\cdot)}(\R^N)$ we have
\begin{enumroman}
\item \label{esti1} $\displaystyle\int_{\R^N}\abs{u(x)}^{p(x)-1}\abs{v(x)} dx\le p^+\sigma(u)\norm[p(\cdot)]{v}$;
\item \label{esti2} $\abs{\rho(u)-\rho(v)}\le(p^+)^2(\sigma(u)+\sigma(v))\norm[p(\cdot)]{u-v}$.
\end{enumroman}
\end{lemma}
\begin{proof}
We prove \ref{esti1}. Taking $a = \big(|u(x)|/\norm[p(\cdot)]{u}\big)^{p(x) - 1},\, b = |v(x)|/\norm[p(\cdot)]{v},\, q = p(x)$ in the well-known Young's inequality
\begin{equation}\label{young}
ab \le \Big(1 - \frac{1}{q}\Big) a^{q/(q-1)} + \frac{1}{q}\, b^q \quad a, b \ge 0,\, q> 1,
\end{equation}
and integrating over $\R^N$ gives (note that $\norm[p(\cdot)]{u}^{p(x)-1}\leq\sigma(u)$ in $\R^N$)
\begin{align*}
\frac{1}{\sigma(u)\norm[p(\cdot)]{v}}\int_{\R^N}\abs{u(x)}^{p(x)-1}\abs{v(x)} dx &\le
\int_{\R^N} \abs{\frac{|u(x)|}{\norm[p(\cdot)]{u}}}^{p(x)-1} \frac{|v(x)|}{\norm[p(\cdot)]{v}} dx \\
&\le \int_{\R^N} \Big(1-\frac{1}{p(x)}\Big)\abs{\frac{u(x)}{\norm[p(\cdot)]{u}}}^{p(x)} dx+\int_{\R^N} \abs{\frac{v(x)}{\norm[p(\cdot)]{v}}}^{p(x)} \frac{dx}{p(x)} \\
&\le p^+,
\end{align*}
the last inequality following from \eqref{one}.
\vskip4pt
\noindent
Now we prove \ref{esti2}. Taking $a = |u(x)|,\, b = |v(x)|,\, q = p(x)$ in the elementary inequality
\[|a^q - b^q| \le q(a^{q-1} + b^{q-1})(a - b), \quad a, b \ge 0,\, p > 1,\]
and integrating over $\R^N$ gives
\begin{align*}
\abs{\rho(u)-\rho(v)} &\le \int_{\R^N} \abs{|u(x)|^{p(\cdot)} - |v(x)|^{p(\cdot)}} dx \\
&\le p^+\Big(\int_{\R^N} |u(x)|^{p(x)-1}\big(|u(x)|-|v(x)|\big) dx+\int_{\R^N} |v(x)|^{p(x)-1}\big(|u(x)| - |v(x)|\big) dx\Big) \\
&\le (p^+)^2 (\sigma(u)+\sigma(v))\norm[p(\cdot)]{u-v},
\end{align*}
the last inequality following from \ref{esti1}.
\end{proof}

\noindent
From Lemma \ref{esti} \ref{esti2} it follows that, if $(u_k)$, $(v_k)$ are bounded sequences in $L^{p(\cdot)}(\R^N)$, then there exists $C>0$ such that
\[\abs{\rho(u_k)-\rho(v_k)}\le C\norm[p(\cdot)]{u_k-v_k}, \quad k\in\N.\]
The autonomous case $V(x)\equiv V^\infty$, $p(x)\equiv p^\infty$ represents a limit case for \eqref{ele}. We set
\[\rho^\infty(u) = \int_{\R^N} |u(x)|^{p^\infty}dx, \quad \sigma^\infty(u) = \norm[p^\infty]{u}^{p^\infty - 1}, \quad u \in L^{p^\infty}(\R^N),\]
so Lemma \ref{esti} gives for all $u, v \in L^{p^\infty}(\R^N)$
\begin{equation} \label{sinf}
\int_{\R^N} |u(x)|^{p^\infty - 1}|v(x)|dx \le p^\infty\sigma^\infty(u) \norm[p^\infty]{v},
\end{equation}
\begin{equation} \label{rinf}
|\rho^\infty(u) - \rho^\infty(v)| \le (p^\infty)^2(\sigma^\infty(u) + \sigma^\infty(v)) \norm[p^\infty]{u - v}.
\end{equation}
Morevoer, we have the following asymptotic laws:

\begin{lemma} \label{trans}
If $u \in H^1(\R^N)$, $(y_k)$ is a sequence in $\R^N$, $|y_k| \to \infty$, and $u_k = u(\cdot - y_k)$, then
\begin{enumroman}
\item \label{trans1} $\rho(u_k) \to \rho^\infty(u)$,
\item \label{trans2} $I(u_k) \to I^\infty(u)$,
\item \label{trans3} $J(u_k) \to J^\infty(u)$.
\end{enumroman}
\end{lemma}
\begin{proof}
We prove \ref{trans1}. For all $k\in\N$, the change of variable $z=x-y_k$ gives
\[\rho(u_k) = \int_{\R^N} |u_k(x)|^{p(x)}dx =\int_{\R^N} |u(z)|^{p(z + y_k)}dz.\]
Since $p(\cdot + y_k) \to p^\infty$ by \eqref{asyp} and $|u(z)|^{p(z + y_k)} \le |u(z)|^{p^-} + |u(z)|^{p^+}$, the last integral converges to $\rho^\infty(u)$ by the dominated convergence theorem.
\vskip4pt
\noindent
We prove \ref{trans2}. As in the proof of \ref{trans1}, for all $\gamma>0$
\begin{equation}\label{trans4}
\int_{\R^N} \abs{\frac{u_k(x)}{\gamma}}^{p(x)} \frac{dx}{p(x)}\to \int_{\R^N} \abs{\frac{u(z)}{\gamma}}^{p^\infty} \frac{dz}{p^\infty} = \Big(\frac{\norm[p^\infty]{u}}{\gamma}\Big)^{p^\infty}.
\end{equation}
If $\norm[p(\cdot)]{u_k} \not\to \norm[p^\infty]{u}$, then there exists $\varepsilon_0 > 0$ such that, on a renumbered subsequence, either $\norm[p(\cdot)]{u_k} \le \norm[p^\infty]{u} - \varepsilon_0$ or $\norm[p(\cdot)]{u_k} \ge \norm[p^\infty]{u} + \varepsilon_0$. In the former case, $\norm[p^\infty]{u} \ge \varepsilon_0$ and, taking $\varepsilon_0$ smaller if necessary, we may assume that this inequality is strict. Then
\[\int_{\R^N} \Big(\frac{|u_k(x)|}{\norm[p^\infty]{u} - \varepsilon_0}\Big)^{p(x)} \frac{dx}{p(x)} \le \int_{\R^N} \Big(\frac{|u_k(x)|}{\norm[p(\cdot)]{u_k}}\Big)^{p(x)} \frac{dx}{p(x)} = 1.\]
Passing to the limit as $k\to\infty$, \eqref{trans4} implies
\[\Big(\frac{\norm[p^\infty]{u}}{\norm[p^\infty]{u} - \varepsilon_0}\Big)^{p^\infty}\leq 1,\]
a contradiction. The latter case leads to a similar contradiction.
\vskip4pt
\noindent
Finally we prove \ref{trans3}. We have for all $k\in\N$
\[J(u_k) = \int_{\R^N} \left(|\nabla u_k|^2 + V(x)\, u_k^2\right) dx = \int_{\R^N} \left(|\nabla u|^2 + V(z + y_k)\, u^2\right) dz,\]
where $z = x - y_k$. Since $V(\cdot + y_k) \to V^\infty$ by \eqref{asyv} and $V \in L^\infty(\R^N)$, the last integral converges to $J^\infty(u)$ by the dominated convergence theorem.
\end{proof}

\section{A compactness result}

\noindent
In this section we prove that $\restr{J}{\M}$ satisfies \PS{c} whenever $c\in\R$ lies below a certain threshold level. The main technical tool that we will use for handling the convergence matters is the following profile decomposition of Solimini \cite{S} for bounded sequences in $H^1(\R^N)$.

\begin{proposition} \label{sol}
Let $(u_k)$ be a bounded sequence in $H^1(\R^N)$, and assume that there is a constant $\delta > 0$ such that, if $u_k(\cdot + y_k) \rightharpoonup w \ne 0$ on a renumbered subsequence for some sequence $(y_k)$ in $\R^N$ with $|y_k| \to \infty$, then $\norm{w} \ge \delta$. Then there exist $m \in \N$, $w^{(1)},\ldots w^{(n)} \in H^1(\R^N)$, and sequences $(y^{(1)}_k),\ldots (y^{(n)}_k)$ in $\R^N$, $y^{(1)}_k = 0$ for all $k \in \N$, $w^{(n)} \ne 0$ for all $2\le n \le m$, such that, on a renumbered subsequence,
\begin{enumroman}
\item\label{sol1} $u_k(\cdot + y^{(n)}_k) \rightharpoonup w^{(n)}$;
\item\label{sol2} $\big|y^{(n)}_k - y^{(l)}_k\big| \to \infty$ for all $n \ne l$;
\item\label{sol3} $\displaystyle\sum_{n=1}^m\, \|w^{(n)}\|^2 \le \liminf_k \norm{u_k}^2$;
\item\label{sol4} $\displaystyle u_k - \sum_{n=1}^m\, w^{(n)}(\cdot - y^{(n)}_k) \to 0$ in $L^q(\R^N)$ for all $q \in (2,2^\ast)$.
\end{enumroman}
\end{proposition}

\noindent
Since $y^{(1)}_k=0$ for all $k\in\N$, by \ref{sol2} we have $|y^{(n)}_k|\to\infty$ for all $2\le n\le m$. Moreover, by \ref{sol4}, \eqref{ppm} and \cite[Theorem 3.3.11]{DHHR}, we also have
\begin{equation}\label{convpx}
u_k - \sum_{n=1}^m\, w^{(n)}(\cdot - y^{(n)}_k) \to 0 \quad\mbox{in $L^{p(\cdot)}(\R^N)$.}
\end{equation}
Next we show that the sublevel sets of $\widetilde{J}$ are bounded. Set
\[\widetilde{J}^a = \bgset{u \in \M : \widetilde{J}(u) \le a}, \quad a \in \R.\]

\begin{lemma} \label{bdd}
For all $a \in \R$, $\widetilde{J}^a$ is bounded.
\end{lemma}
\begin{proof}
By \eqref{asyv}, we can find $R > 0$ such that $|V(x) - V^\infty| <V^\infty/2$ for all $|x|>R$. So, $V^\infty/2 - V \le 0$ outside the ball $B_R(0)$. For all $u \in \widetilde{J}^a$, we have
\begin{align*}
\frac{\norm{u}^2}{2} &\le \int_{\R^N} \Big(|\nabla u|^2 + \frac{V^\infty}{2}u^2\Big) dx = a+\int_{\R^N} \Big(\frac{V^\infty}{2} - V(x)\Big) u^2dx \\
&\le a+\int_{B_R(0)} \Big(\frac{V^\infty}{2} - V(x)\Big) u^2dx \le a+\Big(\frac{V^\infty}{2} + \norm[\infty]{V}\Big) \int_{B_R(0)} u^2dx.
\end{align*}
Since $u \in \M$, and $L^{p(\cdot)}(B_R(0))$ is continuously embedded in $L^2(B_R(0))$ (see \cite[Corollary 3.3.4]{DHHR}), the last integral is bounded. So, $\widetilde{J}^a$ is bounded.
\end{proof}

\noindent
Now, let $(u_k)$ be a \PS{c}-sequence for $\restr{J}{\M}$ for some $c\in\R$, namely $\widetilde{J}(u_k)\to c$ and $\widetilde{J}'(u_k)\to 0$. Then, there exists a sequence $(\mu_k)$ in $\R$ such that $J'(u_k)-\mu_k I'(u_k)\to 0$ in $H^{-1}(\R^N)$. So we have
\begin{equation} \label{conveq}
- \Delta u_k + V(x)u_k = \frac{\mu_k}{2\rho(u_k)}|u_k|^{p(x) - 2} u_k+ \o(1).
\end{equation}
Testing \eqref{conveq} with $u_k$, we easily get $\mu_k/2\to c$. Besides, since $u_k\in\M$, by \eqref{ppm} we have $p^-\le\rho(u_k)\le p^+$, whence, on a renumbered subsequence, $\rho(u_k)\to\rho_0$ for some $\rho_0\in[p^-,p^+]$.
\vskip4pt
\noindent
We prove some technical properties of $(u_k)$:

\begin{lemma} \label{pslim}
Let $(u_k)$ be as above and $w\in H^1(\R^N)$:
\begin{enumroman}
\item \label{pslim1} if $u_k \rightharpoonup w$ on a renumbered subsequence, then
\[-\Delta w + V(x)w = \frac{c}{\rho_0}|w|^{p(x) - 2}w;\]
\item \label{pslim2} if $u_k(\cdot + y_k) \rightharpoonup w$ on a renumbered subsequence for some sequence $(y_k)$ in $\R^N$ with $|y_k| \to \infty$, then
\[\Delta w + V^\infty w = \frac{c}{\rho_0}|w|^{p^\infty - 2}w.\]
\end{enumroman}
\end{lemma}
\begin{proof}
We prove \ref{pslim1}. By the density of $C^\infty_0(\R^N)$ in $H^1(\R^N)$, it suffices to show that
\begin{equation} \label{pslim3}
\int_{\R^N} \big(\nabla w \cdot \nabla v + V(x)wv\big)dx = \frac{c}{\rho_0} \int_{\R^N} |w|^{p(x) - 2}wv dx, \quad v \in C^\infty_0(\R^N).
\end{equation}
We have $\supp v\subset\Omega$ for some bounded domain $\Omega\subset\R^N$. Testing \eqref{conveq} with $v$ gives
\[
\int_\Omega\big(\nabla u_k \cdot \nabla v + V(x)u_k v\big)dx=\frac{\mu_k}{2\rho(u_k)}\int_\Omega|u_k|^{p(x)-2}u_k v dx+\o(1).
\]
Since $V \in L^\infty(\R^N)$ and $u_k \rightharpoonup w$, we have
\[\int_\Omega\big(\nabla u_k \cdot \nabla v + V(x)u_k v\big)dx\to\int_{\R^N} \big(\nabla w \cdot \nabla v + V(x)wv\big)dx.\]
Besides, $\mu_k/(2\rho(u_k))\to c/\rho_0$. Finally, by compactness of the embedding $H^1(\Omega) \hookrightarrow L^{p^+}(\Omega)$, on a renumbered subsequence we have $u_k\to w$ in $L^{p^+}(\Omega)$ and $u_k(x)\to w(x)$ a.e. in $\Omega$. By \eqref{ppm}, \eqref{young} we have a.e. in $\Omega$
\[|u_k|^{p(x) - 1}|v| \le \big(1 + |u_k|^{p^+ - 1}\big)|v| \le |v| + \Big(1 - \frac{1}{p^+}\Big) |u_k|^{p^+} + \frac{1}{p^+}|v|^{p^+},\]
hence by the generalized dominated convergence theorem
\[\int_\Omega|u_k|^{p(x)-2}u_k w dx\to\int_{\R^N} |w|^{p(x) - 2}wv dx,\]
which proves \ref{pslim1}.
\vskip4pt
\noindent
We prove \ref{pslim2}. As above, we only need to show that for all $v \in C^\infty_0(\R^N)$
\begin{equation} \label{pslim4}
\int_{\R^N} \big(\nabla w \cdot \nabla v + V^\infty wv\big)dx = \frac{c}{\rho_0} \int_{\R^N} |w|^{p^\infty - 2}wv dx.
\end{equation}
We have $\supp v\subset\Omega$ for some bounded domain $\Omega\subset\R^N$. Testing \eqref{conveq} with $v(\cdot - y_k)$ and making the change of variable $z = x - y_k$ gives for all $k\in\N$
\[\int_{\R^N} \big(\nabla \widetilde{u}_k \cdot \nabla v + V(z + y_k)\, \widetilde{u}_k v\big)dz = \frac{\mu_k}{2 \rho(u_k)} \int_{\R^N} |\widetilde{u}_k|^{p(z + y_k) - 2}\widetilde{u}_k v dz + \o(1),\]
where $\widetilde{u}_k = u_k(\cdot + y_k)$. Since $\widetilde{u}_k \rightharpoonup w$ and $V(\cdot + y_k) \to V^\infty$ uniformly on $\Omega$, we have
\[\int_{\R^N} \big(\nabla \widetilde{u}_k \cdot \nabla v + V(z + y_k)\, \widetilde{u}_k v\big)dz\to\int_{\R^N} \big(\nabla w \cdot \nabla v + V^\infty wv\big)dx.\]
Besides, $\mu_k/(2\rho(u_k))\to c/\rho_0$. Finally, exploiting again \eqref{ppm}, \eqref{young} as in the proof of \ref{pslim1} we have, on a renumbered subsequence,
\[\int_{\R^N} |\widetilde{u}_k|^{p(z + y_k) - 2}\widetilde{u}_k v dz\to\int_{\R^N} |w|^{p^\infty - 2}wv dx,\]
which proves \ref{pslim2}.
\end{proof}

\noindent
The main result of this section is the following extension to the variable exponent case of the compactness lemma of Benci and Cerami \cite[Lemma 3.1]{BC}.

\begin{proposition}\label{pd}
Let $(u_k)$ be a {\em \PS{c}}-sequence for $\widetilde{J}$, $c\in\R$. Then there exist $m \in \N$, $w^{(1)},\ldots w^{(n)} \in H^1(\R^N)$, and sequences $(y^{(1)}_k),\ldots (y^{(n)}_k)$ in $\R^N$, $y^{(1)}_k = 0$ for all $k \in \N$, $w^{(n)} \ne 0$ for all $2\le n \le m$, such that, on a renumbered subsequence, $\rho(u_k)\to\rho_0$ for some $\rho_0 \in [p^-,p^+]$, and
\begin{enumroman}
\item\label{pd1} $u_k(\cdot + y^{(n)}_k) \rightharpoonup w^{(n)}$;
\item\label{pd2} $\big|y^{(n)}_k - y^{(l)}_k\big| \to \infty$ for all $n \ne l$;
\item\label{pd3} $\displaystyle\sum_{n=1}^m\, \|w^{(n)}\|^2 \le \liminf_k \norm{u_k}^2$;
\item\label{pd4} $-\Delta w^{(1)}+V(x)w^{(1)}=c/\rho_0 |w^{(1)}|^{p(x) - 2}w^{(1)}$;
\item\label{pd5} $-\Delta w^{(n)}+V^\infty w^{(n)}=c/\rho_0 |w^{(1)}|^{p^\infty - 2}w^{(n)}$, $2\le n\le m$;
\item\label{pd6} $J(w^{(1)})=c/\rho_0 \rho(w^{(1)})$;
\item\label{pd7} $J^\infty(w^{(n)})=c/\rho_0 \rho^\infty(w^{(n)})$, $2\le n\le m$;
\item\label{pd8} $\displaystyle\rho(w^{(1)})+\sum_{n=2}^m\rho^\infty(w^{(n)})=\rho_0$;
\item\label{pd9} $\displaystyle J(w^{(1)})+\sum_{n=2}^m J^\infty(w^{(n)})=c$;
\item\label{pd10} $\displaystyle u_k - \sum_{n=1}^m\, w^{(n)}(\cdot - y^{(n)}_k) \to 0$ in $H^1(\R^N)$.
\end{enumroman}
\end{proposition}
\begin{proof}
By Lemma \ref{bdd}, the sequence $(u_k)$ is bounded. Passing to a subsequence, we have $\rho(u_k)\to\rho_0$. We shall apply Proposition \ref{sol}. To this end, set
\[\delta=\Big(\frac{p^- (\lambda_1^\infty)^\frac{p^\infty}{2}}{p^\infty c}\Big)^\frac{1}{p^\infty - 2}>0\]
(in particular, $c>0$). If $u_k(\cdot+y_k)\rightharpoonup w$ in $H^1(\R^N)$, on a renumbered subsequence, for some sequence $(y_k)$ in $\R^N$, $|y_k|\to\infty$ and some $w\neq 0$, then by Lemma \ref{pslim} \ref{pslim2} and the definition of $\lambda_1^\infty$ we have, testing with $w$,
\[\|w\|^2=\frac{p^\infty c}{\rho_0}\norm[p^\infty]{w}^{p^\infty}\le\frac{p^\infty c}{p^-}(\lambda_1^\infty)^{-\frac{p^\infty}{2}}\|w\|^{p^\infty},\]
hence $\|w\|\ge\delta$. Then, by Proposition \ref{sol} there exist $w^{(n)}$, $(y^{(n)}_k)$ ($1\le n\le m$) and a renumbered subsequence $(u_k)$ satisfying \ref{pd1}-\ref{pd3} and
\[u_k - \sum_{n=1}^m\, w^{(n)}(\cdot - y^{(n)}_k) \to 0 \quad\mbox{in $L^q(\R^N)$, $2<q<2^\ast$.}\]
By \eqref{convpx}, the convergence above also holds in $L^{p(\cdot)}(\R^N)$. From \ref{pd}, \ref{pd2} and Lemma \ref{pslim} we deduce \ref{pd4} and \ref{pd5}. Further, testing \ref{pd4} with $w^{(1)}$ and \ref{pd5} with $w^{(n)}$ yields \ref{pd6} and \ref{pd7}, respectively.
\vskip4pt
\noindent
We prove now \ref{pd8}. Set for all $k\in\N$
\[w_k = \sum_{n=1}^m w^{(n)}(\cdot - y^{(n)}_k).\]
Since $\norm[p(\cdot)]{u_k - w_k} \to 0$ by \eqref{convpx} and $\norm[p(\cdot)]{u_k} = 1$, $\norm[p(\cdot)]{w_k} \to 1$. Since $\rho(u_k) \to \rho_0$, then $\rho(w_k) \to \rho_0$ by Lemma \ref{esti} \ref{esti2}. Let $\varepsilon > 0$. Since $C^\infty_0(\R^N)$ is dense in $L^{p(\cdot)}(\R^N)$, there exists $\widetilde{w}^{(1)} \in C^\infty_0(\R^N)$ such that $\norm[p(\cdot)]{w^{(1)} - \widetilde{w}^{(1)}} < \varepsilon$, and since $C^\infty_0(\R^N)$ is dense in $L^{p^\infty}(\R^N)$, for all $2\le n\le m$ there exists $\widetilde{w}^{(n)} \in C^\infty_0(\R^N)$ such that $\norm[p^\infty]{w^{(n)} - \widetilde{w}^{(n)}} < \varepsilon$. Let
\[\widetilde{w}_k = \sum_{n=1}^m\, \widetilde{w}^{(n)}(\cdot - y^{(n)}_k).\]
By \ref{pd2} and Lemma \ref{trans} \ref{trans2} we have
\begin{align*}
\norm[p(\cdot)]{w_k - \widetilde{w}_k} & \le \sum_{n=1}^m\norm[p(\cdot)]{w^{(n)}(\cdot - y^{(n)}_k) - \widetilde{w}^{(n)}(\cdot - y^{(n)}_k)} \\
& \to \norm[p(\cdot)]{w^{(1)} - \widetilde{w}^{(1)}} + \sum_{n=2}^m\norm[p^\infty]{w^{(n)} - \widetilde{w}^{(n)}} \le m\varepsilon.
\end{align*}
So,
\[\limsup_k \norm[p(\cdot)]{w_k - \widetilde{w}_k} \le m\varepsilon.\]
Since $\norm[p(\cdot)]{w_k} \to 1$ and $\rho(w_k) \to \rho_0$, then by Lemma \ref{esti} \ref{esti2} we can find a constant $C>0$ such that
\[\limsup_k |\rho(\widetilde{w}_k) - \rho_0| \le C \varepsilon.\]
On the other hand, for all sufficiently large $k$, the (compact) supports of $\widetilde{w}^{(n)}(\cdot - y^{(n)}_k)$ are pairwise disjoint by \ref{pd2} and hence
\[\rho(\widetilde{w}_k) = \sum_{n=1}^m\, \rho(\widetilde{w}^{(n)}(\cdot - y^{(n)}_k)) \to \rho(\widetilde{w}^{(1)}) + \sum_{n=2}^m \rho^\infty(\widetilde{w}^{(n)})\]
by \ref{pd2} and Lemma \ref{trans} \ref{trans1}. So
\[
\abs{\rho(\widetilde{w}^{(1)}) + \sum_{n=2}^m\, \rho^\infty(\widetilde{w}^{(n)}) - \rho_0} \le C \varepsilon.
\]
By Lemmas \ref{esti} \ref{esti2} and \eqref{rinf} we have
\[\abs{\rho(w^{(1)}) - \rho(\widetilde{w}^{(1)})} \le C \varepsilon, \qquad \abs{\rho^\infty(w^{(n)}) - \rho^\infty(\widetilde{w}^{(n)})} \le C \varepsilon, \quad 2\le n\le m\]
by Lemmas \ref{esti} and \ref{trans}. Since $\varepsilon > 0$ is arbitrary, \ref{pd8} follows.
\vskip4pt
\noindent
Adding \ref{pd6} and \ref{pd7} and substituting \ref{pd8}, we get \ref{pd9}.
\vskip4pt
\noindent
We conclude by proving \ref{pd10}. Set $v_k=u_k-w_k$ and $\widetilde{u}_k=u_k-w^{(1)}$ for all $k\in\N$. Note that both $(\widetilde{u}_k)$ and $(v_k)$ are bounded in $H^1(\R^N)$. By \eqref{conveq}, \ref{pd4} and \ref{pd5} we have
\begin{align*}
- \Delta v_k + V^\infty v_k &= (V^\infty - V(x))\widetilde{u}_k + \frac{\mu_k}{2 \rho(u_k)}|u_k|^{p(x) - 2}u_k \\
&-\frac{c}{\rho_0} \Big(|w^{(1)}|^{p(x) - 2} w^{(1)} + \sum_{n=2}^m|w^{(n)}(x - y^{(n)}_k)|^{p^\infty - 2}w^{(n)}(c - y^{(n)}_k)\Big) + \eta_k,
\end{align*}
for a sequence $(\eta_k)$ in $H^{-1}(\R^N)$ with $\eta_k\to 0$ in $H^{-1}(\R^N)$. Testing with $v_k$ and using Lemma \ref{esti} \ref{esti1} and \eqref{sinf} gives
\[\norm{v_k}^2 \le \int_{\R^N} \abs{(V(x) - V^\infty)\widetilde{u}_k v_k} dx + C \big((\sigma(u_k) + 1) \norm[p(\cdot)]{v_k} + \norm[p^\infty]{v_k}\big)+ \o(\norm{v_k}).\]
Since $\norm{u_k}$ is bounded, so are $\norm{\widetilde{u}_k}$, $\sigma(u_k)$, and $\norm{v_k}$. If $\norm{v_k} \not\to 0$, then there exists $\varepsilon_0 > 0$ such that, on a renumbered subsequence, $\norm{v_k} \ge \varepsilon_0$. By \eqref{asyv}, there exists $R > 0$ such that
\[\int_{B_R(0)^c} \abs{(V(x) - V^\infty)\widetilde{u}_k v_k} dx \le 2\sup_{x \in B_R(0)^c}|V(x) - V^\infty| \norm[2]{\widetilde{u}_k} \norm[2]{v_k} \le \frac{\varepsilon_0^2}{2}.\]
Then, from the equation above, H\"{o}lder inequality, Proposition \ref{sol} \ref{sol4} and \eqref{convpx} we get
\[\frac{\varepsilon_0^2}{2} \le C \big(\norm[L^2(B_R(0))]{v_k} + \norm[p^\infty]{v_k} + \norm[p(\cdot)]{v_k}\big) + \o(1) \to 0,\]
a contradiction. Thus, \ref{pd10} is proved.
\end{proof}

\noindent
Now we prove that $\restr{J}{\M}$ satisfies \PS{c} whenever $c$ lies below a threshold level:

\begin{theorem} \label{ps}
Assume that $V \in L^\infty(\R^N)$ satisfies \eqref{asyv} and $p \in C(\R^N)$ satisfies \eqref{ppm} and \eqref{asyp}. Then $\widetilde{J}$ satisfies {\em \PS{c}} for all
\begin{equation} \label{pst}
c < \left(\frac{p^-}{p^\infty}\right)^{2/p^\infty} \lambda_1^\infty.
\end{equation}
\end{theorem}
\begin{proof}
Let $c$ satisfy condition \eqref{pst} and let $u_k \in \M$ be a \PS{c} sequence for $\widetilde{J}$.
Then $u_k$ admits a renumbered subsequence that satisfies the conclusions of Proposition \ref{pd}.
Let us set
$$
t_1 = \rho(w^{(1)})/\rho_0,\qquad
t_n = \rho^\infty(w^{(n)})/\rho_0,\quad \text{for $n = 2,\dots,m$}.
$$
Then
\begin{equation} \label{4.2}
\sum_{n=1}^m\, t_n = 1
\end{equation}
by \ref{pd8} of Proposition \ref{pd}, so each $t_n \in [0,1]$, and $t_n \ne 0$ for $n \ge 2$. For $n = 2,\dots,m$,
\[
c\, t_n = J^\infty(w^{(n)}) \ge \lambda_1^\infty \norm[p^\infty]{w^{(n)}}^2 = \lambda_1^\infty \left(\frac{\rho_0\, t_n}{p^\infty}\right)^{2/p^\infty} \ge \lambda_1^\infty \left(\frac{p^-\, t_n}{p^\infty}\right)^{2/p^\infty}
\]
by \ref{pd7} of Proposition \ref{pd} and \eqref{lambda1inft}, so
\[
t_n \ge \left[\frac{\lambda_1^\infty}{c} \left(\frac{p^-}{p^\infty}\right)^{2/p^\infty}\right]^{p^\infty/(p^\infty - 2)} > 1
\]
by \eqref{pst}. Then \eqref{4.2} implies $m = 1$ and hence $u_k \to w^{(1)}$ in $H^1(\R^N)$ by $(x)$ of Proposition \ref{pd}.
\end{proof}

\medskip

\section{Proofs of main theorems}

\subsection{Proof of Theorem \ref{main}}
Since the sublevel sets of $\widetilde{J}$ are bounded by Lemma \ref{bdd} and $\widetilde{J}$ is clearly bounded on bounded sets, $\lambda_1 > - \infty$. To see that $\lambda_1 \le \lambda_1^\infty$, let $w_1^\infty$ be the minimizer of $J^\infty$ on $\M^\infty$ mentioned in the introduction, $y_k \in \R^N$, $|y_k| \to \infty$, and $w_k = w_1^\infty(\cdot - y_k)$. Since $w_k/\norm[p(\cdot)]{w_k} \in \M$,
\[
\lambda_1 \le J\Big(\frac{w_k}{\norm[p(\cdot)]{w_k}}\Big) = \frac{J(w_k)}{\norm[p(\cdot)]{w_k}^2} \to \frac{J^\infty(w_1^\infty)}{\norm[p^\infty]{w_1^\infty}^2} = \lambda_1^\infty
\]
by \ref{trans2} and \ref{trans3} of Lemma \ref{trans}, so $\lambda_1 \le \lambda_1^\infty$. Assume now that \eqref{ineq} holds.
Since $\widetilde{J}$ satisfies the Palais-Smale condition at the level $\lambda_1$ by Theorem \ref{ps}, it has a minimizer
$w_1$ by a standard argument. Then $|w_1|$ is a minimizer too and
hence we may assume that $w_1 \ge 0$. Since $w_1 \ne 0$, then $w_1 > 0$ by the strong maximum principle.
Observe that a solution $u\in H^1(\R^N)$ of \eqref{ele} satisfies $-\Delta u=g(x,u)$ and
$$
\Big|\frac{g(x,u)}{u}\Big|\leq C+C|u|^{p(x)-2}\leq C+C|u|^{\frac{4}{N-2}},\quad \text{for some $C=C(V,N,p^+)>0$}.
$$
Then by standard regularity theory, $u$ is of class $C^1(\R^N)$, see e.g.\ \cite[Appendix B]{S1}.
\qed

\medskip

\subsection{Proof of Theorem \ref{local}}
Since $e^{- \psi(|x|)}/\norm[p(\cdot)]{e^{- \psi(|x|)}} \in \M$,
\begin{equation} \label{4.4}
\lambda_1 \le J\Big(\frac{e^{- \psi(|x|)}}{\norm[p(\cdot)]{e^{- \psi(|x|)}}}\Big) = \frac{J(e^{- \psi(|x|)})}{\norm[p(\cdot)]{e^{- \psi(|x|)}}^2}.
\end{equation}
By virtue of condition \eqref{plocal},
\begin{equation} \label{4.5}
\rho(e^{- \psi(|x|)}) = \int_{\R^N} e^{- \psi(|x|)\, p(x)}\, dx \ge e^a \int_{B_R(0)^c} e^{- p^\infty \psi(|x|)}\, dx.
\end{equation}
It follows from \eqref{4.4}, \eqref{modnorm} and \eqref{4.5} that \eqref{ineq} holds if $a > 0$ is sufficiently large. \qed

\subsection{Proof of Corollary \ref{symm}}
If $N\geq 3$, $V$ and $p$ are radially symmetric in $\R^N$,
we can get some symmetry properties of minimizers by applying the results of Mari\c{s} \cite{M}. We can equivalently define
\[
\mathcal{M}=\left\{u\in H^1(\R^N) : \ \int_{\R^N}|u(x)|^{p(x)}\frac{dx}{p(x)}=1\right\}.
\]
For any hyperplane $\Pi$ through $0$, splitting $\R^N$ in two half-spaces $\Pi^+$ and $\Pi^-$, and all $u\in H^1(\R^N)$ we define functions $u_{\Pi^+},u_{\Pi^-}:\R^N\to\R$ by setting
\[
u_{\Pi^+}(x) = \begin{cases}
u(x), & \mbox{if $x\in\Pi^+\cup\Pi$}\\[5pt]
u(2P_{\Pi}(x)-x), & \mbox{if $x\in\Pi^-$}
\end{cases},
\]
\[
u_{\Pi^-}(x) = \begin{cases}
u(x), & \mbox{if $x\in\Pi^-\cup\Pi$}\\[5pt]
u(2P_{\Pi}(x)-x), & \mbox{if $x\in\Pi^+$},
\end{cases}
\]
where $P_{(\cdot)}$ is the orthogonal projection from $\R^N$ to an affine submanifold $(\cdot)$. Clearly $u_{\Pi^\pm}\in H^1(\R^N)$, so hypothesis {\bf A1} of \cite{M} is satisfied.
Since  $u$ is of class $C^1(\R^N)$, so hypothesis {\bf A2} holds as well. By \cite[Theorem 1]{M} we learn that,
for every minimizer $w$ of $\tilde J$, there exists a line $L$ through $0$ such that $w(x)=\tilde w(P_L(x),|x-P_L(x)|)$
for all $x\in\R^N$, for a convenient function $\tilde w:L\times\R^+\to\R$.
\qed

\end{document}